\documentclass[10pt]{article}

 \usepackage{amsfonts, epsfig, amsmath, amssymb, color,amsthm}
\usepackage{mathrsfs}
 
\usepackage[english]{babel}

\usepackage{enumerate}
 
\usepackage{tikz}
\usetikzlibrary{decorations.pathmorphing} 
\usetikzlibrary{matrix} 
\usetikzlibrary{arrows} 
\usetikzlibrary{calc} 

\usepackage{pgfplots}
\pgfplotsset{compat=1.13}
\usepgfplotslibrary{fillbetween}

\usetikzlibrary{datavisualization}
\usetikzlibrary{datavisualization.formats.functions}

\tikzstyle{block} = [draw,rectangle,thick,minimum height=2em,minimum width=2em]
\tikzstyle{sum} = [draw,circle,inner sep=0mm,minimum size=2mm]
\tikzstyle{connector} = [->,thick]
\tikzstyle{line} = [thick]
\tikzstyle{branch} = [circle,inner sep=0pt,minimum size=1mm,fill=black,draw=black]
\tikzstyle{guide} = []
\tikzstyle{snakeline} = [connector, decorate, decoration={pre length=0.2cm,
                         post length=0.2cm, snake, amplitude=.4mm,
                         segment length=2mm},thick, magenta, ->]

\tikzset{>=latex}

\newcommand{\mbR}{\mathbb{R}}
\newcommand{\mbZ}{\mathbb{Z}}
\newcommand{\mbN}{\mathbb{N}}

\newcommand{\E}{\mathbf{E}}
\renewcommand{\P}{\mathbf{P}}

\newcommand{\1}{1\!\!\!\;{\rm I}}

\usepackage{color}

\renewcommand {\epsilon}{\varepsilon}

\theoremstyle{plain}
\newtheorem{thm}{Theorem}[section]
\newtheorem{lem}[thm]{Lemma}
\newtheorem{prp}[thm]{Proposition}
\newtheorem{cor}[thm]{Corollary}

\theoremstyle{definition}
\newtheorem{rem}[thm]{Remark}
\newtheorem{exa}[thm]{Example}

\DeclareMathSymbol{\ophi}{\mathalpha}{letters}{"1E}

\newcommand{\e}{\varepsilon}
\newcommand{\ve}{\varepsilon}

\renewcommand{\phi}{\varphi}

\newcommand{\be}{\begin{equation}}
\newcommand{\bel}{\begin{equation}\label}
\newcommand{\ee}{\end{equation}}
\newcommand{\ben}{\begin{equation*}}
\newcommand{\een}{\end{equation*}}

\newcommand{\ba}{\begin{equation}\begin{aligned}}
\newcommand{\ea}{\end{aligned}\end{equation}}

\DeclareMathOperator{\sgn}{sgn}

\DeclareMathOperator{\Law}{Law}

\renewcommand{\i}{\mathrm{i}}
\newcommand{\ex}{\mathrm{e}}
\newcommand{\di}{\mathrm{d}}

\newcommand{\cH}{\mathcal{H}}

\newcommand{\cX}{\mathcal{X}}
\newcommand{\cY}{\mathcal{Y}}
\newcommand{\cZ}{\mathcal{Z}}

\newcommand{\rF}{\mathscr{F}}

\newcommand{\bF}{\mathbb{F}}

\newcommand{\bR}{\mathbb{R}}

\newfont{\cyrfnt}{wncyr10}
\def\J3{\cyrfnt{\rm \u{\cyrfnt I}}}
\def\j3{\cyrfnt{\rm \u{\cyrfnt i}}}

\allowdisplaybreaks[4]

\numberwithin{equation}{section}

\begin{document}
 
\title{Limit behaviour of random walks on $\mbZ^m$ with two-sided membrane}

\date{\today} 
 
\author{Victor  Bogdanskii\footnote{National Technical University of Ukraine 
``Igor Sikorsky Kyiv Polytechnic Institute'',
ave.\ Pobedy 37, Kiev 03056, Ukraine;  vbogdanskii@ukr.net},\ \
 Ilya Pavlyukevich\footnote{Institute for Mathematics, Friedrich Schiller University Jena, Ernst--Abbe--Platz 2,
07743 Jena, Germany; ilya.pavlyukevich@uni-jena.de}\ \  and 
Andrey Pilipenko\footnote{Institute of Mathematics, National Academy of Sciences of Ukraine, Tereshchenkivska Str.\ 3, 01601, Kiev, 
Ukraine} \footnote{National Technical University of Ukraine 
``Igor Sikorsky Kyiv Polytechnic Institute'',
ave.\ Pobedy 37, Kiev 03056, Ukraine; pilipenko.ay@gmail.com}
}

\maketitle
 
\begin{abstract}
 
We study Markov chains on $\mathbb Z^m$, $m\geq 2$, 
that behave like a standard symmetric random walk outside of the hyperplane (membrane) $H=\{0\}\times \mathbb Z^{m-1}$.  
The transition probabilities on the membrane $H$ are periodic and also depend on the incoming direction to $H$, what makes the membrane
$H$ two-sided. Moreover, sliding along the membrane is allowed. We show that the natural scaling limit of such Markov chains
is a $m$-dimensional diffusion whose first coordinate is a skew Brownian motion
and the other $m-1$ coordinates is a Brownian motion with a singular drift controlled by the local time of the first coordinate at $0$. 
In the proof we utilize a 
 martingale characterization of the Walsh Brownian motion and determine the effective permeability and slide direction. 
Eventually, a similar convergence theorem is established for the one-sided membrane without slides and random iid transition probabilities.
\end{abstract}



\section{Introduction}

A multidimensional Brownian motion is a fundamental stochastic process that describes an idealized mathematical model of
a free physical diffusion in a homogeneous medium. 
Having in mind the observations of pollen made my Brown or a later kinetic theory of gases, we can interpret a diffusion as a 
collective motion of independent random walkers whose distribution density in space obeys the isotropic Gaussian distribution.

However in real physical or biological systems, the space is often separated into compartments by \emph{membranes} 
that impede or facilitate the passage of the walker and create an anisotropy in the walkers' collective motion.
From the physical point of view, a membrane is a thin slice of a material whose diffusivity 
is different from the diffusivity of the environment.
Diffusions through membranes are observed in biological tissues where they 
control the transport of ions, water molecules and gases, or in porous and composite materials. 

A rigorous mathematical justification of the interpretation of a diffusion as a limit of scaled random walks is given by the 
Functional Central Limit Theorem (FCLT, the Donsker--Prokhorov invariance principle). 
For $m\geq 1$, let $\{\ex_1,\dots,\ex_m\}$ be a standard basis in $\bR^m$. Consider a symmetric 
random walk $ Z=(Z(n))_{n\geq 0}$ on $\mbZ^m$ defined by identical one-step transition probabilities
\ba
\P\Big( Z(n+1)- Z(n)=\pm \ex_k\Big| Z(0),\dots,  Z(n)\Big)= \frac{1}{2m},\quad k=1,\dots m,\quad n\geq 0.
\ea
Then the FCLT yields the weak convergence in the uniform topology
\ba
\Big(\frac{ Z([nt])}{\sqrt{n}}\Big)_{t\geq 0}\Rightarrow \Big(\frac{1}{\sqrt m} W(t)\Big)_{t\geq 0},
\ea
where $W$ is a standard $m$-dimensional Brownian motion.

In this paper we consider a novel class of scaling limits of random walks 
in the presence of extended spatially non-homogeneous one- or two-sided barrier (membrane, interface). More precisely, in the $m$-dimensional
space we consider an semi-permeable hyperplane that separates the space into two half-space compartments. 
The random walker will perform the symmetric random walk outside of the membrane, however the probability of the passage through the membrane
will depend on the hitting position of the membrane by the walker and on its incoming direction. Thus the probability to penetrate the 
membrane will depend of the fact whether the walker has reached it from the ``right'' of from the ``left''.
Furthermore we will assume that the membrane has a spatially regular structure according to one of the following setting. First, 
we assume that the penetration probabilities are periodic in space, so that the membrane
reminds of a two-sided fabric. Alternatively, the membrane can imitate a ``random'' perforated surface so that the
penetration probabilities are iid random variables. We will show the scaled limit of such Markov chains
is a skew Brownian motion
in the direction perpendicular to the membrane.
The other coordinates converge to a standard Brownian motion, maybe, with a singular drift given by the local time of the first
coordinate at the origin.

The proofs of these results are purely probabilistic and employ methods of homogenization and dynamics of singular differential equations.
Our results give a path-wise picture of the diffusion through a two-side semi-permeable interface. In the physical language this corresponds to the 
Langevin--Smoluchowsky approach to diffusions. It should be emphasized that physical papers (see, e.g., Novikov et al.~\cite{novikov2011random}, Grebenkov et al.~\cite{grebenkov2014exploringI}, Moutal and Grebenkov \cite{moutal2019diffusion} and references therein) devoted to similar problems use analytical 
methods, mainly the analysis of the Fokker--Planck equation. 

One-dimensional locally perturbed random walks were considered from different points of view in 
Harrison and Shepp \cite{HShepp-81}, Minlos and Zhizhina \cite{minlos1997limit}, Pilipenko and Pryhod'ko \cite{pilipenko2012limit},
Pilipenko and Prikhod'ko \cite{pilipenko2015signular}, Pilipenko and Sakhanenko
\cite{pilipenko2015impurity}, Ngo and Paign\'e \cite{NgoPeigne19}. In this paper, we use the multidimensional martingale characterization approach 
previously considered in  Iksanov and Pilipenko\cite{IksanovPilipenko2016} in dimension one.

It should be noted that  if transition probabilities of a multidimensional random walk are 
perturbed on a finite set or on a hyperplane of co-dimension 2, then under some natural assumptions 
its scaling limit is a Brownian motion, see 
Sz\'asz and Telcs
\cite{szasz1981random},
Yarotskii
\cite{yarotskii2001central}, and
Paulin and Sz\'asz
\cite{paulin2010locally}.

We also refer the reader's attention to the following related mathematical works: 
the monograph by Portenko \cite{portenko1990generalized} and the research
papers by
Lejay
\cite{lejay2016snapping},
Mandrekar and Pilipenko
\cite{mandrekar2016brownian},
Aryasova and Pilipenko
\cite{aryasova2009onbrownian},
Iksanov et al.\
\cite{iksanov2021functional},
 Pilipenko and Khomenko
\cite{pilipenko2017zero}, 
Pilipenko
\cite{pilipenko2017functional}, and
Pilipenko and Prykhodko
\cite{pilipenko2014jump}.

\medskip

\noindent
\textbf{Notation.} The weak convergence in the Skorokhod space $D([0,\infty),\bR^m;J_1)$ is denoted by $\Rightarrow$. It should be noted, however, 
that all limit processes in this paper are continuous. The convergence in distribution of random variables is denoted by $\stackrel{\di}{\to}$.

\medskip

\noindent
\textbf{Acknowledgements.} I.P.\ and A.P.\ acknowledge  support
 by the DFG project PA 2123/7-1.
 A.P.\ thanks the Institute of Mathematics
of the FSU Jena for hospitality. 
V.B.\ and A.P.\ acknowledge support by the National Research Foundation of Ukraine
 (project 2020.02/0014 ``Asymptotic regimes of perturbed random walks: on the edge of modern and
 classical probability'').

\section{Two-sided periodic membrane: the model and the main result\label{s:twosided}}
 
Let $m\geq 2$ and let $\{\ex_1,\dots,\ex_m\}$ be a standard basis in $\mbR^m$. 
Consider a Markov chain $Z=\{Z(n)\}_{n\geq 0}$ on $\mbZ^m$ that behaves as 
a simple random walk outside of the hyperplane $H:=\{0\}\times \mbZ^{m-1}$, i.e.,\
for all $k=1,\dots,m$
\be
\label{eq:transition}
\P\Big(Z({n+1})=z\pm \ex_k\Big| Z(n)=z\Big)=\frac{1}{2m}, \ z\notin H.
\ee
For each $n\geq 1$, we denote the first coordinate of 
the process $Z$ by $X$, and the other $(m-1)$  coordinates by $Y$ so that
$Z=(X, Y)$.

We will interpret $H$ as a semipermeable two-sided membrane that may let a particle pass from one half-space 
to another with certain probabilities that can depend on the crossing direction. Moreover, the particle can 
``slide''
along the membrane.

The membrane has to be homogeneous, i.e.,\ the transition 
probabilities are periodic in space.

Notice that if the membrane is two-sided, then $Z$ is not a Markov chain, generally. Indeed, 
its position upon leaving the membrane is determined by both current the location on the membrane and on the particle's incoming direction. 
Hence, in order to introduce a Markov structure 
we have to enlarge the state space by splitting the membrane $H$ into two parts $\cH^-$ and $\cH^+$ corresponding to its ``right'' and ``left'' sides.

To formalize the setting, we consider the Markov chain $\mathcal Z=(\mathcal Z(n))_{n\geq 0}$
on the state space
\ba
\{\dots, -3, -2,-1,-0, +0, 1,2,3,\dots\}\times \mbZ^{m-1}.
\ea
We also write $\cZ=(\cX,Y)$, so that the process $Y$ coincides with the process $Y$ of the original random walk $Z$ whereas the process 
$\cX$ is defined on the enlarged space  $\{\dots, -3, -2,-1,-0, +0, 1,2,3,\dots\}$.

Denote the ``right'' and ``left'' hand side of the membrane by
\ba
\cH^-:=\{- 0\}\times \mbZ^{m-1},\quad \cH^+:=\{+0\}\times \mbZ^{m-1},
\ea
and we set $\mathcal{H}:=\cH^-\cup \cH^+$.

We assume that transition probabilities of $\mathcal Z$
outside of $\cH$ satisfy
\eqref{eq:transition}, namely
\be
\P\Big(\mathcal Z({n+1})=z\pm \ex_k\, \Big| \, \mathcal Z(n)=z\Big)=\frac{1}{2m}, \ z\notin \cH,
\ee
where we agree that for each $y\in\mbZ^{m-1}$
\ba
(1,y)-\ex_1=(+0,y)\quad &\text{and}\quad (-1,y)+\ex_1=(-0,y),\\
(+0,y)+ \ex_1=(1,y)\quad &\text{and}\quad (-0,y)+ \ex_1=(1,y),\\
(+0,y)- \ex_1=(- 1,y)\quad &\text{and}\quad (-0,y)- \ex_1=(- 1,y).
\ea
Suppose also that the perturbed random walk does not stay on the membrane, i.e.,\
\ba
\P\Big(\cZ({n+1})\in \{\pm 1\}\times \mbZ^{m-1} \, \Big| \, \cZ(n)=z\Big)=1, \quad z\in \cH.
\ea
However it is allowed that upon hitting the membrane (from the left or from the right) the particle can ``slide''
along the membrane, so that its $y$-coordinate changes.
We assume that transition probabilities from $\cH$ have periodic structure.

\noindent 
\textbf{Notation.} 
For $k_2,\dots,k_{m}\geq 1$ fixed, let $U$ denote the box in $\mbZ^{m-1}$ defined by
\ba
U:=[0,k_2-1]\times \cdots\times [0,k_m-1].
\ea
For $y\in \mbZ^{m-1}$, $y=(y_2,\dots,y_{m})$, and $j\in U$, $j=(j_2,\dots ,j_{m})$, we say that
\ba
y\equiv  j\quad \text{if}\quad y_i \ (\textrm{mod}\ k_i) = j_i ,\  i=2,\dots,m.
\ea  
For each $j\in U$ we set
\ba
H_j^\pm:= \{ (\pm 0,y)\in H^\pm \colon   y\equiv   j\}.
\ea
Without loss of generality let us assume that the random walk $\cZ$ starts on the hyperplane $\cH$, i.e.,\ $\cX(0)=\pm 0$.

The following are our key assumptions concerning the transition probabilities of the random walk in the membrane. 

\medskip

\noindent
\textbf{A}$_\text{periodic}$.  Periodicity of the transition probabilities. We assume that there exist $k_2,\dots,k_m\geq 1$ such that for all 
$l_2,\dots,l_m\in\mbZ$, for all $z_0\in \cH$, for all $z_1\in \{- 1,+1\}\times \mbZ^{m-1}$
\ba
\label{eq:periodic_membrane}
\P\Big(\cZ({n+1})=z_1\, &\Big| \, \cZ(n)=z_0\Big)\\
&=\P\Big(\cZ(n+1)=z_1+ k_2 l_2 \ex_2+ \cdots +k_m l_m \ex_m \, \Big| \, \cZ(n)=z_0+ k_2 l_2 \ex_2+ \cdots +k_m l_m \ex_m\Big).
\ea

\noindent 
\textbf{A}$_\gamma$. To describe the transitions through the membrane, we denote by $0=\tau_0<\tau_1<\cdots$ 
the successive arrivals of $\cZ$ to $\cH$ or, equivalently, of 
$\cX$ to $\{- 0,+0\}$.

On the finite state space $\{-0, +0\}\times U$ 
we consider an auxiliary embedded process $\hat \cY=(\hat \cY(n))_{n\geq 0}$ as follows.
We set
\ba
\label{e:hatY}
\hat \cY(n)=
\begin{cases} 
(-0,j), \ \text{ if } Y(\tau_n)\equiv j\in U\text{ and }    X(\tau_n+1)=-1,\\
(+0, j), \ \text{ if } Y(\tau_n)\equiv j\in U\text{ and }  X(\tau_n+1)=1.
\end{cases}
\ea
It follows from the strong Markov property of $\cZ$ that $\hat \cY$ is a Markov chain. 
It is easy to see that all the states of the set  $\{-0, +0\}\times U $ are connected for $\hat \cY$. 
Hence there is a unique stationary distribution 
\ba
\pi=\{\pi_{\{- 0\}\times j}, \pi_{\{+ 0\}\times j}\}_{j\in U}
\ea
of $\hat\cY$ on $\{-0, +0\}\times U $.
By the strong law of large numbers for Markov chains, for each $j\in U$
\be
\label{eq:statPi}
\lim_{n\to\infty} \frac{1}{n}\sum_{k=0}^n\1(\hat \cY(k)=\{\pm0\}\times j) = \pi_{\{\pm 0\}\times j} \ \ \text{a.s.}
\ee
With the help of the stationary distribution $\pi$ we introduce the \emph{effective permeability}
\ba
\label{e:gamma}
\gamma&:=\sum_{j\in U}(\pi_{\{+0\}\times j}- \pi_{\{-0\}\times j})\in[-1,1].
\ea

\noindent 
\textbf{A}$_c$. Finally, we describe the ``slides'' along the membrane. We denote
\ba
\label{e:c}
\alpha_{\pm, j}&:=\E\Big[Y(1)-Y(0)\,\Big|\, \cX(0)=\pm 0, Y(0)=j\Big]\in\bR^{m-1},\quad j\in U,\\
\ea
the mean slide sizes along the ``right'' of the ``left'' side of the membrane, 
and assume that they are finite. 
Introduce the \emph{effective slide}   as
\ba
c&:=\E_\pi\Big[Y(1)-Y(0)\Big]
=\sum_{j\in U}\Big(\pi_{\{+0\}\times j}\cdot \alpha_{+,j}+ \pi_{\{-0\}\times j}\cdot \alpha_{-,j}\Big)\in\bR^{m-1}.
\ea

Now we are ready to formulate the main result of the paper. 
Let $W_X$ and $W_Y$ be independent standard Brownian motions in $\bR$ and $\bR^{m-1}$ respectively.

For $\gamma\in[-1,1]$ defined in \eqref{e:gamma}, we consider the skew 
Brownian motion $X^{\gamma}$ that is a unique strong solution of the SDE 
\ba
\label{eq:Skew}
X^{\gamma}(t) =\gamma L(t) + W_X(t) , \quad t\geq 0.
\ea
where $L$ is the symmetric two-sided local time of $X^\gamma$ at zero, see Harrison and Shepp \cite{HShepp-81}. Other characterizations
and properties of the skew Brownian motion can be found in the review by Lejay \cite{Lejay-06}.

Furthermore, let
\ba
\label{eq:Y}
Y^{c}(t) :=c L(t) + W_Y(t) , \quad t\geq 0,
\ea
i.e.,\ the process $Y^c$ is a $(m-1)$-dimensional Brownian motion that slides in the direction $c$ in the time instants when $X^\gamma$ touches zero.
Note that \eqref{eq:Y} is not a stochastic differential equation because the process $L$ is already determined in \eqref{eq:Skew}.

For the original process $Z=(X,Y)$, denote the rescaled continuous time processes
\ba
\label{e:XY}
X_n(t):=\frac{X([nt])}{\sqrt{n}}, \quad  Y_n(t):=\frac{Y([nt])}{\sqrt{n}},\quad t\geq 0.
\ea

\begin{thm}
\label{thm:periodic_membrane}
Let $Z=(X,Y)$ be a perturbed random walk satisfying the preceding assumptions. Then for any initial value 
$Z(0)\in\mbZ^m$,
the 
weak convergence 
holds true:
\ba
(X_n,Y_n)\Rightarrow \frac{1}{\sqrt m} ( X^{\gamma}, Y^c ),\quad n\to\infty,
\ea
where the processes $X^\gamma$ and $Y^c$ are defined in \eqref{eq:Skew} and \eqref{eq:Y}, and $\gamma$ and $c$ are defined in 
\eqref{e:gamma} and \eqref{e:c} respectively.
\end{thm}

The crux of the Theorem is transparent. Away from the membrane $H$, the limiting process $(X^\gamma,Y^c)$ coincides with the 
standard $m$-dimensional Brownian motion $(W_X,W_Y)/\sqrt{m}$. The perturbation of the transition probabilities on the two-sided membrane 
results in the appearance of a singular drift in the direction $x$ perpendicular to the membrane. 
Hence the $x$-coordinate of the limiting process becomes a skew Brownian motion with the effective permeability parameter $\gamma\in[-1,1]$.
The $y$-coordinates also become a singular drift in the effective sliding direction $c\in\bR^{m-1}$. This drift equals 
 zero as long as 
the limiting process stays away from the membrane. However upon hitting the membrane, the limiting process performs a singular ``sliding'' 
in the direction $c$ controlled by the local time at zero of the $x$-coordinate. 
Note that in the limit, the two-sided membrane structure disappears, and 
the limiting process $( X^{\gamma}, Y^c )$ is Markov in $\bR^m$.

We illustrate Theorem \ref{thm:periodic_membrane} by an example.

\begin{exa}\label{exa:1_periodic}
Consider a perturbed random walk $Z=(X,Y)$ on the plane $\mbZ^2$, $m=2$, 
separated by a two-sided membrane $\cH$. Assume that the membrane is 
$2$-periodic, i.e.,\ $k_1=2$ and the set $U=\{0,1\}$. 

Assume that upon hitting the membrane from the ``right'' the random walk
  gets reflected in the even points and surely penetrates the membrane in the odd points.
 The penetration probabilities from the ``left''
  are given at Fig.~\ref{f:fig} a).
  
 Let
\ba
\tau=\inf\{k\geq 0\colon \cZ(n)=(\cX(n),Y(n))\in \cH \}
\ea
be the first hitting time of the Markov chain $\cZ$ of the two-sided membrane $\cH$.
Since away of the membrane $\cH$, the increments of $\cZ$ coincide with the
 increments of a translation invariant 
two-dimensional symmetric random walk $Z$, we can easily 
calculate the probabilities of hitting the membrane in an even or an odd point:
\ba
\label{e:alpha}
\alpha&:=\P_{(1,0)}(Y(\tau)\equiv 0)=\P_{(1,1)}(Y(\tau)\equiv 1)=\P_{(-1,0)}(Y(\tau)\equiv 0)=\P_{(-1,1)}(Y(\tau)\equiv 1)=2-\sqrt 2,\\
1-\alpha&:=\P_{(1,0)}(Y(\tau)\equiv 1)=\P_{(1,1)}(Y(\tau)\equiv 0)=\P_{(-1,0)}(Y(\tau)\equiv 1)=\P_{(-1,1)}(Y(\tau)\equiv 0)=\sqrt 2-1.
\ea
With the help of \eqref{e:alpha} we calculate the transition probabilities of the embedded Markov chain $\hat \cY$ on $\{-0,+0\}\times\{0,1\}$:
\ba
\mathbb P=
    \begin{array}{c|cccc}
            & (-0,0) & (-0,1)& (+0,0) &(+0,1)\\\hline
 \vphantom{1^{|}}    (-0,0) & (1-p)(1-\alpha)  & (1-p)\alpha & p\alpha & p(1-\alpha)\\
                     (-0,1) &  p\alpha& p(1-\alpha)& (1-p)(1-\alpha) & (1-p)\alpha \\ 
                     (+0,0) & 0 & 0 & \alpha& 1-\alpha\\
                     (+0,1) & 1-\alpha & \alpha &0 &0 
\end{array}   
\ea
Solving the forward Kolmogorov equation $(\mathbb P^T -\text{Id}) \pi=0$ we obtain the stationary law $\pi$ of $\hat \cY$:
\ba
\pi_{(-0,0)} &=  \frac{(1-\alpha) (1-\alpha +(2 \alpha -1) p)}{(1-\alpha) (2 \alpha +1)+(2 \alpha -1) p^2}  ,\quad 
\pi_{(-0,1)}=  \frac{(1-\alpha) \alpha }{(1-\alpha) (2 \alpha +1)+(2 \alpha -1) p^2},\quad \\
\pi_{(+0,0)}&= \frac{\alpha  (1-\alpha +(2 \alpha -1) p^2)}{(1-\alpha) (2 \alpha +1)+(2 \alpha -1) p^2} ,\quad 
\pi_{(+0,1)}=  \frac{(1-\alpha) (\alpha -(2 \alpha -1) (1-p) p)}{(1-\alpha) (2 \alpha +1)+(2 \alpha -1) p^2} ,
\ea
and calculate the effective permeability and the effective slide according to \eqref{e:gamma} and \eqref{e:c}:
\ba
\label{e:gammac}
\gamma&=\frac{(2 \alpha -1) (\alpha  (2 p-1)+(p-1)^2 )}{(1-\alpha) (2 \alpha +1)+(2 \alpha -1) p^2} ,\\
c&=-\frac{(1-\alpha) ((2 \alpha -1) p+1)} {(1-\alpha) (2 \alpha +1)+(2 \alpha -1) p^2}.
\ea

\begin{figure}
\begin{center}
\begin{tikzpicture}[scale=1]
\draw[gray](-2,-2.5) -- (-2,3.5);
\draw[gray](-1,-2.5) -- (-1,3.5);
\draw[gray]( 0,-2.5) -- ( 0,3.5);
\draw[gray]( 1,-2.5) -- ( 1,3.5);
\draw[gray]( 2,-2.5) -- ( 2,3.5);

\draw[gray](-2.5,-2) -- (2.5,-2);
\draw[gray](-2.5,-1) -- (2.5,-1);
\draw[gray](-2.5, 0) -- (2.5, 0);
\draw[gray](-2.5, 1) -- (2.5, 1);
\draw[gray](-2.5, 2) -- (2.5, 2);
\draw[gray](-2.5, 3) -- (2.5, 3);

\draw[gray, ultra thick](-0.1,-1.5) -- (-0.1,-0.5);
\draw[gray, ultra thick] (0.1,-1.5) --  (0.1,-0.5);
\draw[black, ultra thick] (-0.1,-2.5) -- (-0.1,-1.5);
\draw[black, ultra thick]  (0.1,-2.5) --  (0.1,-1.5);

\draw[black, ultra thick](-0.1,-0.5) -- (-0.1,0.5);
\draw[black, ultra thick] (0.1,-0.5) --  (0.1,0.5);
\draw[gray, ultra thick] (-0.1, 0.5) -- (-0.1,1.5);
\draw[gray, ultra thick]  (0.1, 0.5) --  (0.1,1.5);

\draw[black, ultra thick](-0.1, 1.5) -- (-0.1,2.5);
\draw[black, ultra thick] (0.1, 1.5) --  (0.1,2.5);
\draw[gray, ultra thick] (-0.1, 2.5) -- (-0.1,3.5);
\draw[gray, ultra thick]  (0.1, 2.5) --  (0.1,3.5);

\filldraw[black]  (-0.1,-2.0) circle (2pt);
\filldraw[black]  ( 0.1,-2.0) circle (2pt);
\filldraw[gray] (-0.1,-1.0) circle (2pt);
\filldraw[gray] ( 0.1,-1.0) circle (2pt);

\filldraw[black] (-0.1,0.0) circle (2pt);
\filldraw[black] ( 0.1,0.0) circle (2pt);
\filldraw[gray]  (-0.1,1.0) circle (2pt);
\filldraw[gray]  ( 0.1,1.0) circle (2pt);

\filldraw[black] (-0.1,2.0) circle (2pt);
\filldraw[black] ( 0.1,2.0) circle (2pt);
\filldraw[gray]  (-0.1,3.0) circle (2pt);
\filldraw[gray]  ( 0.1,3.0) circle (2pt);

\draw[black, dashed, ultra thick,->] (.1,0) .. controls +(up:5mm) and +(up:5mm) .. (1,0) node[anchor=south west] {$1$};
 
\draw[black, very thick,->] (-.1,0) .. controls +(down:5mm) and +(down:5mm) .. (1,0) node[anchor=north west] {$p$};
\draw[black, very thick,->] (-.1,0) .. controls +(down:0mm) and +(down:0mm) .. (-1,-1) node[anchor=north east] {$1-p$};

\draw[gray, dashed ,ultra thick,->] (.1,1) .. controls +(up:8mm) and +(up:7mm) .. (-1,1) node[anchor=south east] {$1$};
\draw[gray, very thick,->] (-.1,1)  .. controls +(down:5mm) and +(down:5mm) .. (1,1) node[anchor=north west] {$1-p$};
\draw[gray, very thick,->] (-.1,1) .. controls +(down:0mm) and +(down:0mm) .. (-1,0) node[anchor=north east] {$p$};

\draw[black]  (-2,-2) node[anchor=south east] {$-2$};
\draw[black]  (-2,-1) node[anchor=south east] {$-1$};
\draw[black]  (-2,0) node[anchor=south east] {$0$};
\draw[black]  (-2,1) node[anchor=south east] {$1$};
\draw[black]  (-2,2) node[anchor=south east] {$2$};
\draw[black]  (-2,3) node[anchor=south east] {$3$};

\draw[black]  (-2,-2) node[anchor=north east] {$-2$};
\draw[black]  (-1,-2) node[anchor=north east] {$-1$};
\draw[black]  (-.1,-2) node[anchor=north east] {$-0$};
\draw[black]  (.1,-2) node[anchor=north west] {$+0$};
\draw[black]  (1,-2) node[anchor=north west] {$1$};
\draw[black]  (2,-2) node[anchor=north west] {$2$};

\draw[black]  (2,3.6) node[anchor=north west] {a)};
\end{tikzpicture}
\hspace{2cm}
\begin{tikzpicture}[scale=1]
\draw[gray](-2,-2.5) -- (-2,3.5);
\draw[gray](-1,-2.5) -- (-1,3.5);
\draw[gray]( 0,-2.5) -- ( 0,3.5);
\draw[gray]( 1,-2.5) -- ( 1,3.5);
\draw[gray]( 2,-2.5) -- ( 2,3.5);

\draw[gray](-2.5,-2) -- (2.5,-2);
\draw[gray](-2.5,-1) -- (2.5,-1);
\draw[gray](-2.5, 0) -- (2.5, 0);
\draw[gray](-2.5, 1) -- (2.5, 1);
\draw[gray](-2.5, 2) -- (2.5, 2);
\draw[gray](-2.5, 3) -- (2.5, 3);

\draw[black, ultra thick](-0.1,-1.5) -- (-0.1,-0.5);
\draw[gray, ultra thick] (0.1,-1.5) --  (0.1,-0.5);
\draw[gray, ultra thick] (-0.1,-2.5) -- (-0.1,-1.5);
\draw[black, ultra thick]  (0.1,-2.5) --  (0.1,-1.5);

\draw[gray, ultra thick](-0.1,-0.5) -- (-0.1,0.5);
\draw[black, ultra thick] (0.1,-0.5) --  (0.1,0.5);
\draw[black, ultra thick] (-0.1, 0.5) -- (-0.1,1.5);
\draw[gray, ultra thick]  (0.1, 0.5) --  (0.1,1.5);

\draw[gray, ultra thick](-0.1, 1.5) -- (-0.1,2.5);
\draw[black, ultra thick] (0.1, 1.5) --  (0.1,2.5);
\draw[black, ultra thick] (-0.1, 2.5) -- (-0.1,3.5);
\draw[gray, ultra thick]  (0.1, 2.5) --  (0.1,3.5);

\filldraw[gray]  (-0.1,-2.0) circle (2pt);
\filldraw[black]  ( 0.1,-2.0) circle (2pt);
\filldraw[black] (-0.1,-1.0) circle (2pt);
\filldraw[gray] ( 0.1,-1.0) circle (2pt);

\filldraw[gray] (-0.1,0.0) circle (2pt);
\filldraw[black] ( 0.1,0.0) circle (2pt);
\filldraw[black]  (-0.1,1.0) circle (2pt);
\filldraw[gray]  ( 0.1,1.0) circle (2pt);

\filldraw[gray] (-0.1,2.0) circle (2pt);
\filldraw[black] ( 0.1,2.0) circle (2pt);
\filldraw[black]  (-0.1,3.0) circle (2pt);
\filldraw[gray]  ( 0.1,3.0) circle (2pt);

\draw[black, dashed, ultra thick,->] (.1,0) .. controls +(up:5mm) and +(up:5mm) .. (1,0) node[anchor=south west] {$1$};
 
\draw[gray, very thick,->] (-.1,0) .. controls +(down:5mm) and +(down:5mm) .. (1,0) node[anchor=north west] {$1-p$};
\draw[gray, very thick,->] (-.1,0) .. controls +(down:0mm) and +(down:0mm) .. (-1,-1) node[anchor=north east] {$p$};

\draw[gray, dashed ,ultra thick,->] (.1,1) .. controls +(up:8mm) and +(up:7mm) .. (-1,1) node[anchor=south east] {$1$};
\draw[black, very thick,->] (-.1,1)  .. controls +(down:5mm) and +(down:5mm) .. (1,1) node[anchor=north west] {$p$};
\draw[black, very thick,->] (-.1,1) .. controls +(down:0mm) and +(down:0mm) .. (-1,0) node[anchor=north east] {$1-p$};

\draw[black]  (-2,-2) node[anchor=south east] {$-2$};
\draw[black]  (-2,-1) node[anchor=south east] {$-1$};
\draw[black]  (-2,0) node[anchor=south east] {$0$};
\draw[black]  (-2,1) node[anchor=south east] {$1$};
\draw[black]  (-2,2) node[anchor=south east] {$2$};
\draw[black]  (-2,3) node[anchor=south east] {$3$};

\draw[black]  (-2,-2) node[anchor=north east] {$-2$};
\draw[black]  (-1,-2) node[anchor=north east] {$-1$};
\draw[black]  (-.1,-2) node[anchor=north east] {$-0$};
\draw[black]  (.1,-2) node[anchor=north west] {$+0$};
\draw[black]  (1,-2) node[anchor=north west] {$1$};
\draw[black]  (2,-2) node[anchor=north west] {$2$};

\draw[black]  (2,3.6) node[anchor=north west] {b)};
\end{tikzpicture} 
\end{center}   
\caption{Two models with a two-sided 2-periodic semi-permeable membrane. The left-hand sides of the membranes a) and b)
are shifted by 1 along the $y$-axis. 
\label{f:fig}}
\end{figure}
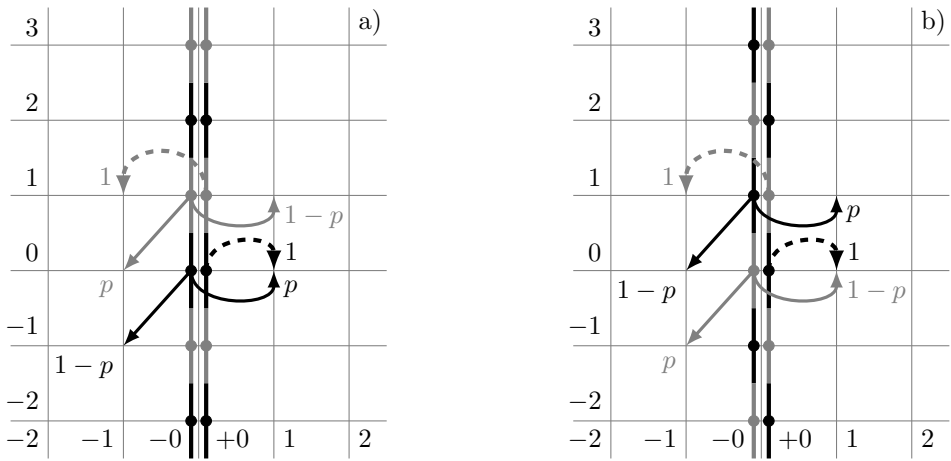
\end{exa}
\begin{exa}
Consider the random walk from Example \ref{exa:1_periodic}, where transition
probabilities from the ``right'' side of the membrane are the same, and transition
probabilities from 
 the ``left'' side are shifted by 1 along the $y$-axis, see Fig.~\ref{f:fig} b). 

The results for this model  are obtained immediately by interchanging $p$ and $1-p$ in 
\eqref{e:gammac}. Notice that parameters of effective permeability and effective slide are different,
i.e., they depends on mutual disposition of the ``left'' and the ``right'' sides of the membrane.
\end{exa}
The idea of the proof of the Theorem \ref{thm:periodic_membrane} consists in a decomposition of the process $\cZ$ into excursions 
starting and ending on the membrane $\cH$. The excursions have a probability law of excursions of the symmetric
random walk.

If the membrane is homogeneous ($k_2=\cdots=k_m=1$), then
it is well know that the limit of $x$ coordinate   is a skew Brownian motion,
see Harrison and Shepp \cite{HShepp-81}.  
To control a slide in $y$ direction we only have to control the number of visits to 0 of $x$-coordinate of the walk.
However, if the membrane is periodic with a non-trivial period, then sign of the $x$-coordinate of an excursion is selected in accordance with the 
transition probabilities \eqref{eq:periodic_membrane}, i.e., 
its sign depends on the position of the random walk 
at the last visit to the membrane. The slide along $y$ depends on that position too.

Due to the periodicity assumption \eqref{eq:periodic_membrane}, 
there is $d=2|U|:=2k_2\cdots k_m$ different types of excursions between consecutive
visits of the membrane.  
In order to treat this number of excursions at the same time we have to consider a natural
 generalization of the skew Brownian motion, 
namely the Walsh's Brownian motion. Hence we will show that the family of $d$ one-sided 
random walks converge to a Walsh's Brownian motion. This will allow us to 
 derive the effective permeability and sliding parameters.

The paper is organized as follows. In the next Section we introduce Walsh's Brownian motion and give its convenient  
realization as a $d$-dimensional stochastic process that takes values on the positive coordinate half-axes. We will also formulate 
two martingale characterizations of Walsh's Brownian motion. Section \ref{s:pm} is devoted to the proof of Theorem 
\ref{thm:periodic_membrane}. In the final Section \ref{s:rm} we show how our method 
can be applied to a model of a one-sided membrane that possesses some ergodic properties.

\section{Walsh's Brownian motion and its martingale characterizations\label{s:WBM}}

Walsh's Brownian motion (WBM) was introduced in the Epilogue of Walsh \cite{walsh1978diffusion} as a diffusion on $d$ rays on a two-dimensional plane
with the common origin.
On each ray WBM, is a standard one-dimensional Brownian motion that however can change the ray upon hitting the origin, i.e.,\ each ray is characterized
by a weight $p_i>0$, $i=1,\dots,d$, $p_1+\cdots+p_d=1,$ that heuristically can be understood as a probability to go on the ray number $i$.
Hence, the conventional WBM is a process $X$ on the plane expressed in polar coordinates as $X=(R_t,\theta_t)$ where $R$ is the reflecting 
Brownian motion and $\theta$ is a random process taking values on the set of 
$d$ angles on $[0,2\pi)$ and being constant during each excursion of $R$ from $0$. This representation has been used in various 
works on WBM including Barlow et al.\ \cite{barlow1989walsh},
Freidlin and Sheu
\cite{freidlin2000diffusion},
Hajri 
\cite{hajri2011stochastic},
Hajri and  Touhami
\cite{hajri2014ito}, and 
 Karatzas and Yan
\cite{karatzas2019semimartingales}.

In this paper prefer to embed the WBM into a $d$-dimensional Euclidean space as it was indicated in Walsh \cite{walsh1978diffusion}.
To this purpose, let $E$ be the union of non-negative coordinate half-axes in $\bR^d$, i.e.,\ 
\ba
E=\{x\in\bR^d\colon x_i\geq 0\text{ and } x_ix_j=0,\ i\neq j,\  i,j=1,\dots, d\}.
\ea
We fix probabilities $p_1,\dots,p_d> 0$, $p_1+\cdots+p_d=1$, and also denote $q_i=1-p_i$, $i=1,\dots,d$.
Let also $(\Omega,\rF,\bF,\P)$ be a filtered probability space satisfying the usual hypotheses. Then adopting 
the Markovian characterization of the WBM from Barlow et al.\ \cite{barlow1989walsh}, we say that WBM is a continuous 
Feller Markov process $X=(X_1,\dots,X_d)$ on $E$
with the one-dimensional laws given by
\ba
\E_0 \ex^{\lambda_1 X_1(t)+\cdots +\lambda_d X_d(t) }&=\sum_{k=1}^d p_k \E_0 \ex^{\lambda_k |W(t)|},\\
\E_x \ex^{\lambda_1 X_1(t)+\cdots+\lambda_d X_d(t)}&=
\displaystyle 
\E_{x_j}\Big[   \1(t<\tau_0)\ex^{\lambda_j W(t)}\Big]
+ \sum_{k=1}^d  p_k \E_{x_j}\Big[      \1(t\geq\tau_0)\ex^{\lambda_k |W(t)|} \Big]\\
&= \E_{x_j}\Big[\sum_{k=1}^d p_k \ex^{\lambda_k |W(t)|}\Big]
 + \E_{x_j}\Big[ \1(t<\tau_0) \Big( \ex^{\lambda_j W(t)}- 
\sum_{k=1}^d p_k\ex^{\lambda_k W(t)}\Big)\Big],\\
x&=(0,\dots,x_j,\dots,0),\ x_j>0, \lambda\in\mathbb C^d,
\ea
  $W$ is a standard Brownian motion and $\tau_0=\inf\{t\geq 0\colon W_t=0\}$.
  Notice that the last expectation can be considered as the expectation of a killed Brownian motion.

In Barlow et al.\ \cite{barlow1989walsh}, the authors also gave the martingale characterization of the WBM realized as a process on the plane. In terms of 
the $d$-dimensional realization $X$ of the WBM, their characterization takes the following form.

\begin{thm}[Propositon 3.1 and Theorem 3.2 in Barlow et al.\ \cite{barlow1989walsh}]
\label{Walsh1} 
Let $X=(X_1(t),\dots,X_d(t))_{t\geq 0}$ be an adapted continuous process. Then $X$ is a WBM
with parameters $p_1,\dots,p_d>0$ if and only if it satisfies 
the following conditions:
\begin{enumerate}
\item $X_i(t)\geq 0$ and $X_i(t)X_j(t)=0$ for all $i\neq j$ and $t\geq 0$;

\item for each $i=1,\dots,d$ the process
\ba
N_i(t):=q_i X_i(t)-p_i\sum_{j\neq i}X_j(t) ,\quad t\geq 0,
\ea
is a continuous martingale with respect to $\bF$;
\item 
for each $i=1,\dots,d$ the process 
\ba
\label{e:brN}
(N_i(t))^2 -\int_0^t \Big(q_i \1(X_i(s)>0)-p_i  \1(X_i(s)= 0)\Big)^2 \,\di s ,\quad t\geq 0,
\ea
is a continuous martingale with respect to $\bF$.
\end{enumerate}
\end{thm}
Note that since the product of the indicator functions in \eqref{e:brN} is identically zero, we have
\ba
\langle N_i\rangle_t&=
\int_0^t \Big(q_i \1(X_i(s)>0)-p_i  \1(X_i(s)= 0)\Big)^2 \,\di s\\
&=\int_0^t \Big( q_i^2 \1(X_i(s)>0)+p_i^2 \1(X_i(s)= 0)\Big)\,\di s.
\ea
Furthermore, it is clear, see Lemma 2.2 in Barlow et al.\ \cite{barlow1989walsh}, that the radial process
\ba
R(t):=\sum_{i=1}^d X_i(t)=\max_{1\leq i\leq d} X_i(t)
\ea
is a reflecting Brownian motion, and hence it has a local time at $0$ defined by:
\ba\label{e:loctime}
L^X(t):= L^R(t):=\lim_{\ve\to 0+}\frac{1}{2\e}\int_0^t \1\Big(\max_{1\leq i\leq d} X_i(s) \leq \ve\Big) \,\di s.
\ea
In the following theorem we give an equivalent martingale characterization of the WBM that better fits into the setting of this paper. 

\begin{thm}
\label{Walsh}
Let $X=(X_1(t),\dots,X_d(t))_{t\geq 0}$ and $\nu=(\nu(t))_{t\geq 0}$ be adapted continuous processes.
Then $X$ is a WBM with parameters $p_1,\dots,p_d>0$, and
$\nu$ is the local time of $X$ at $0$
 if and only if they satisfy   
the following conditions:
\begin{enumerate}[a)]
\item $X_i(t)\geq 0$ and $X_i(t)X_j(t)=0$ for all $i\neq j$, $t\geq 0$;

\item $\nu(0)=0$, $\nu$ is nondecreasing a.s., $\int_0^\infty \1(X(s)\neq 0)\, \di \nu(s)=0$ a.s.;

\item the processes $M_1,\dots, M_d$ defined by
\ba
\label{e:Mnu}
M_i(t):=X_i(t)-p_i \nu(t),\quad t\geq 0,
\ea
are continuous square integrable martingales with respect to $\bF$ with the
predictable quadratic variations
\ba
\label{e:brM}
\langle M_i\rangle_t=\int_0^t \1(X_i(s)>0 ) \, \di s.
\ea
\item $\int_0^\infty \1(X(s)=0) \,\di s=0$ a.s.
\end{enumerate}
\end{thm}
\begin{rem}
There is a typo  in \S 3 in Barlow et al.\ \cite{barlow1989walsh} (p. 281).
One should remove the indicator $\1_{r>0}$ in the formula for $h_i(r,\theta)$,
 otherwise the zero process is also a solution of equations (3.2) or (3.3) from their paper.
Note that condition (d) in our Theorem \ref{Walsh} prohibits $0$ to be a sticky point of $X$.
\end{rem}

\begin{proof}
We will show that conditions 1--3 of Theorem \ref{Walsh1} and a)--d) of Theorem \ref{Walsh} are equivalent.

\noindent
i) We show that a)--d) $\Rightarrow$  1--3.

First, it is obvious that a) implies 1.

By c), each process $N_i$ is a martingale as a linear combination of martingales:
\ba
N_i(t)&=q_i X_i(t)-p_i\sum_{j\neq i}X_j(t)\\
&= q_i( M_i(t)+p_i \nu(t))- p_i\sum_{j\neq i}  ( M_j(t)+p_j \nu(t)) \\
&=q_i M_i(t)- p_i \sum_{j\neq i}  M_j(t).
\ea

Observe that quadratic covariations $\langle M_i,M_j\rangle$ vanish for $i\neq j$. Indeed, applying the Cauchy-type inequality for 
quadratic covariations (Proposition 15.10 in Kallenberg \cite{Kallenberg-02}) we get
\ba
\langle M_i,M_j\rangle_t& = \int_0^t \di \langle M_i,M_j\rangle_s= \int_0^t\Big(\1(X_i(s)>0)+\1(X_i(s)=0)
 \Big)\, \di \langle M_i,M_j\rangle_s\\
&\leq \Big(\int_0^t \1(X_i(s)>0)\, \di \langle M_j\rangle_s \cdot  \int_0^t \di \langle M_i\rangle_s   
\Big)^{1/2}
+\Big(\int_{0}^t\1(X_i(s)=0) \,\di \langle M_i\rangle_s\cdot  \int_0^t \di \langle M_j\rangle_s   \Big)^{1/2}\\
&\leq \Big(\int_0^t \1(X_i(s)>0) \1(X_j(s)>0)\, \di s \cdot  t \Big)^{1/2}+\Big(\int_0^t \1(X_i(s)=0)
\1(X_i(s)>0) \,\di s\cdot t  \Big)^{1/2}\\
&=0
\ea
Consequently with the help of c) we obtain the martingale property 3. Indeed, by the It\^o formula we get
\ba
(N_i(t))^2
&=(N_i(0))^2+2\int_0^t N_i(s)\,\di N_i(s) + \langle N_i\rangle_t \\
&=(N_i(0))^2+2q_i \int_0^t N_i(s)\,\di M_i(s) -2p_i\sum_{j\neq i}  \int_0^t N_i(s)\,\di M_j(s) \\
&+ q_i^2  \langle M_i\rangle_t + p_i^2\sum_{j\neq i} \langle M_j\rangle_t.
\ea
We have
\ba
\int_0^t \Big(q_i \1(X_i(s)>0)-p_i \1(X_i(s)= 0)\Big)^2 \,\di s
&=\int_0^t \Big(q_i^2 \1(X_i(s)>0)+p_i^2 \1(X_i(s)= 0)\Big)\,\di s\\
&= q_i^2  \langle M_i\rangle_t + p_i^2\sum_{j\neq i} \langle M_j\rangle_t,
\ea
where we used (d) in the last equality. Hence we get
the property 3. 

Show that the process $\nu$ is the local time of $X$ at zero. Let $L^X$ be the local time of $X$ ar zero and
let
\ba
R(t)=\max_{1\leq i\leq d} X_i(t)= \sum_{i=1}^d X_i(t)
\ea
be the radial part of the WBM $X$. By Lemma 2.2 in Barlow et al.\ \cite{barlow1989walsh}, $R$ is a reflecting Brownian motion 
with a local time $L^R$, and
hence
\ba
R(t)-L^R(t)
\ea
is a Brownian motion.
On the other hand
\ba
R(t)-\nu(t)=\sum_{i=1}^d M_i(t)
\ea
and $Z(t):=\sum_{i=1}^d M_i(t)$ is a continuous martingale with the bracket (we use 3.)
\ba
\langle Z\rangle_t=\sum_{i=1}^d \int_0^t \1(X_i(s)> 0)\,\di s=t,
\ea
thus, $Z$ is a $\bF$-Brownian motion. By the uniqueness of the semimartingale decomposition, $\nu=L^R$. Notice that 
$L^R=L^X$.
 
\noindent 
ii) We show that 1--3 $\Rightarrow$ a)--d). 

Let $X$ be a WBM. Denote $\nu:=L^X$ its local time at 0. The properties a), b) and d) follow immediately. 

For each $i=1,\dots,d$ consider the process
\ba
Y_i(t)= X_i(t)-\sum_{j\neq i}X_j(t) ,\quad t\geq 0.
\ea
and a rescaled martingale 
\ba
\tilde N_i(t):=\frac{1}{p_iq_i}N_i(t)=\frac{1}{p_i}X_i(t)-\frac{1}{q_i}\sum_{j\neq i}X_j(t) ,\quad t\geq 0,
\ea
with the bracket
\ba
\langle \tilde N_i\rangle_t=\int_0^t \Big(\frac{1}{p_i^2} \1(X_i(s)>0)+\frac{1}{q_i^2} \1(X_i(s)= 0)\Big)\,\di s.
\ea
Then obviously
\ba
\1(N_i(s)>0)&=
\1(\tilde N_i(s)>0)= \1(X_i(s)>0)=\1(Y_i(s)>0),\\
\1(N_i(s)<0)&= \1(\tilde N_i(s)<0) =\1(Y_i(s)<0).
\ea
Consider the processes
\ba
W_i(t)&:=\int_0^t \Big( p_i\1(\tilde N_i(s)>0) + q_i \1(\tilde N_i(s)< 0)  \Big)\,\di \tilde N_i(s).
\ea
These process are continuous martingales with the bracket
\ba
\langle W_i\rangle_t =t,
\ea
hence they are Brownian motions. Denoting 
$\sigma_i(x):=\frac{1}{p_i} \1(x>0)+ \frac{1}{q_i}\1(x< 0)$ 
we get that $\tilde N_i$ satisfies the SDE
\ba
\tilde N_i(t)= \tilde N_0+\int_0^t \sigma_i(\tilde N_i(s))\,\di W_i(s) , 
\ea
and hence each $N_i$ is the so-called oscillating Brownian motion, see Keilson and Wellner \cite{keilson1978oscillating}.
By Nakao's theorem, see Nakao \cite{nakao1972pathwise}, 
this SDE has a unique strong 
solution. 
Let $r_i(x)= p_i \1(x>0)+ q_i\1(x< 0)$ so that
\ba
Y_i=r_i(\tilde N_i).
\ea
Repeating literally the calculations from Section 5.2 in Lejay \cite{Lejay-06} we get that
\ba
Y_i(t)=Y_i(0)+ W_i(t)+ (2p_i-1)L^{Y_i}(t),
\ea
where $L^{Y_i}(t)$ is the symmetric local time of ${Y_i}$ at 0.

Hence $Y_i$ is a skew Brownian motion with parameter $(2p_i-1)$. Moreover the radial process $R=|Y_i|=|X|$ is a reflected Brownian motion and 
\ba
R(t)=|Y_i(t)|=R(0)+\int_0^t \sgn (Y_i)\,\di W_i(s)+  L^{Y_i}(t).
\ea
Futhermore,
\ba
R(t)=R(0)+W(t)+L^R(t)
\ea
for some Brownian motion $W$.
From the uniqueness of the decomposition of $R$ as a semimartingale we get that $L^{Y_i}=L^R$ for all $i=1,\dots,d$.

It follows from \cite[Theorem 1.7, Chapter VI]{RevuzYor05} that the right and left 
local times of $Y_i$ at 0 are equal to
\ba
L^{Y_i}_\text{left}(t)= 2(1-p_i) L^{Y_i}(t),\ \ L^{Y_i}_\text{right}(t)= 2p_i L^{Y_i}(t).
\ea
Finally Tanaka's formula yields 
\ba
X_i(t)&=\max\{Y_i(t),0\}= \\
&=\max\{Y_i(0),0\}+  \int_0^t \1(Y_i(s)>0)\,\di W_i(s) + \frac12 L^{Y_i}_\text{right}(t)\\
&= \max\{Y_i(0),0\}+  \int_0^t \1(Y_i(s)>0)\,\di W_i(s) + p_i L^{R}(t),\\
\ea
and thus the process $M_i=X_i-p_i\nu$ is a continuous martingale with the bracket \eqref{e:brM}.
\end{proof}

\begin{cor}
\label{c:SBM}
Let $X=(X_1(t),\dots,X_d(t))_{t\geq 0}$ be a WBM with parameters $p_1,\dots,p_d>0$ and let $I\subseteq\{1,\dots, d\}$. Let
\ba
\gamma=\sum_{i\in I} p_i - \sum_{j\in I^c} p_j= 2\sum_{i\in I} p_i-1\in[-1,1].
\ea
Then the process
\ba
X^\gamma(t):=\sum_{i\in I} X_i(t)- \sum_{j\in I^c} X_j(t),\quad t\geq 0,
\ea
is a skew Brownian motion with the parameter $\gamma$.
\end{cor}
\begin{proof}
Without loss of generality assume that $I=\{1,\dots,k\}$ for some $0\leq k\leq d$. 
By Theorem~\ref{Walsh}, the process
\ba
W(t)&=\sum_{i=1}^k M_i(t)- \sum_{j=k+1}^d M_j(t)
\ea
is a continuous martingale. Taking into account \eqref{e:Mnu} we get
\ba
W(t)&= X^\gamma(t) - \gamma \nu(t).
\ea
Since the local times at $0$ of $X$ and $X^\gamma$ coincide, we get
\ba
W(t)= X^\gamma(t) - \gamma L(t)\\
\ea
where $L$ is the local time of $X^\gamma$ at $0$. The bracket of the martingale $W$ equals 
\ba
\langle W \rangle_t= \sum_{i=1}^d \int_0^t \1(X_i(s)>0)\,\di s = t,
\ea
and thus $W$ is a Brownian motion. In other words,
  $X^\gamma$ satisfies the SDE $X^\gamma(t)=W(t)+ \gamma L(t)$
and thus is a skew Brownian motion, see Harrison and Shepp \cite{HShepp-81}.
\end{proof}

\section{Proof of Theorem \ref{thm:periodic_membrane}\label{s:pm}}

Consider the Markov chain $\cZ=(\cX,Y)$ on the enlarged state space $\{\pm 0,\pm 1,\dots\}\times \mbZ^{m-1}$.
Let us decompose the process $\cX$ into $2|U|$ non-negative excursions parameterized by the elements of the set $\{-0,+0\}\times U$.
By $\sigma_n$, $n\geq 0$, denote 
the time instant of the last visit of the random walk $\cZ$ to $H_0$ before time $n$, i.e.,\
\ba
\sigma_n:=\max\Big\{k\leq n \colon \cX(k)\in \{- 0,+0\}\Big\},\quad \sigma_0=0.
\ea
Consider $2|U|$ processes 
\ba
\label{eq:X_j_notation}
X^{j,+}(n)&:= X(n)\1(Y(\sigma_n)\equiv  j, X(n)>0),\\ 
X^{j,-}(n)&:= |X(n)|\1(Y(\sigma_n) \equiv  j, X(n)<0), \quad j\in U,
\ea
and introduce a $2|U|$-dimensional 
Markov chain $\{(X^{j,+}(n), X^{j,-}(n))_{j\in U}\}_{n\geq 0}$ on $\mathbb N_0^{2|U|}$.  
Observe that
all its coordinates are non-negative, only one coordinate
of the vector $ (X^{j,+}(n), X^{j,-}(n))_{j\in U}$ may be non-zero, and  $X^{j,\pm }(\sigma_n)=0$.

To study the limit behavior of the scaled process $X_n(\cdot)$ defined in \eqref{e:XY} we will prove that the properly scaled $2|U|$-dimensional process
$ (X^{j,+}(n), X^{j,-}(n))_{j\in U}$
converges to a WBM with parameters $\{\pi_{\{+ 0\}\times j},\pi_{\{- 0\}\times j}\}_{j\in U}$.
To this end, 
for each $j\in U$ we decompose the processes $X^{j,+}$ and $X^{j,-}$ into a sum of a martingale and a ``local time in 0'', namely
we set
\ba
\label{eq:X_j_representation}
X^{j,\pm}(n)&=\sum_{k=1}^{n} (X^{j,\pm}(k)-X^{j,\pm}(k-1))  \1(X^{j,\pm}(k-1)>0) \\
&+ \sum_{k=1}^{n}\1(X^{j,\pm}(k-1)=0, X^{j,\pm}(k)=1)\\
&=: M^{j,\pm}(n) + L^{j,\pm}(n),\qquad M^{j,\pm}(0) =L^{j,\pm}(0)=0.
\ea
Since the original process $Z$ is a symmetric random walk outside of $H$, 
it follows from the construction that 
\ba
\E \Big[   (X^{j,\pm}(k) & -X^{j,\pm}(k-1))  \1(X^{j,\pm}(k-1)>0)    \Big| \rF_{k-1} \Big]\\
&= \1(X^{j,\pm}(k-1)>0)\cdot  \E \Big[ X^{j,\pm}(k)-X^{j,\pm}(k-1)  \Big| \rF_{k-1} \Big]=0
\ea
and the sequences $(M^{j,+}(n))_{n\geq 0}$ and $(M^{j,-}(n))_{ n\geq 0}$ are martingales for any $j\in U$
 with respect to filtration $(\rF_n)_{n \geq 0}$ generated by $Z$.
Moreover, since
\ba
\langle M^{j,\pm}  \rangle_n&=\sum_{k=1}^n \E \Big[   (X^{j,\pm}(k)-X^{j,\pm}(k-1))^2 \cdot \1(X^{j,\pm}(k-1)>0)   \Big| \rF_{k-1} \Big]\\
&= \frac1m\sum_{k=1}^n    \1(X^{j,\pm}(k-1)>0) ,
\ea
the sequences 
\ba
(M^{j,\pm}(n))^2- \frac1m \sum_{k=1}^{n} \1(X^{j,\pm}(k-1)>0),\quad n\geq 0,
\ea
are martingales too.

  Set
\ba
\label{e:scaled}
X_{n}^{j,\pm}(t):= \frac{X^{j,\pm}([nt])}{\sqrt n}, \quad  
M_{n}^{j,\pm}(t):= \frac{M^{j,\pm}([nt])}{\sqrt n}, \quad \nu_{n}^{j,\pm}(t):= \frac{L^{j,\pm}([nt])}{\sqrt n}.
\ea

\begin{prp}
\label{p:WBM}
There are continuous processes $\{X_j^\pm\}_{j\in U}$, continuous
martingales $\{M_j^\pm\}_{j\in U}$ and a nondecreasing process $\nu$ such that
\ba
\label{eq:Walsh}
(X_{n}^{j,\pm} , M_{n}^{j,\pm} , \nu_{n}^{j,\pm} )_{j\in U} 
\Rightarrow (X_j^{\pm} , M_j^\pm , \pi^\pm_j \nu )_{j\in U} ,\quad n\to\infty.
\ea
Moreover, the process $\sqrt{m} (X_j^{\pm} , M_j^\pm , \pi^\pm_j \nu )_{j\in U}$
 satisfies conditions of Theorem \ref{Walsh}.
\end{prp}

To prepare the proof of Proposition \ref{p:WBM} we notice that  in problems of this type it is often helpful to start with the 
study of the radial process
\ba
S(n):=\sum_{j\in U} (X^{j,-}(n)+X^{j,+}(n)), \ n\geq 0.
\ea
The process $S=\{S(n)\}_{n\geq 0}$ is a Markov chain on $\mbN\cup\{0\}$
with reflection at $0$ whose steps in $\mbN$ have the distribution
\ba
&\P\Big(S(n+1)=j-1 \Big| S(n)=j\Big) =\P\Big(S(n+1)=j+ 1 \Big| S(n)=j\Big)=\frac{1}{2m}, \\
&\P\Big(S(n+1)=j  \Big| S(n)=j\Big) =1- \frac{1}{m}, \quad j\geq 1,\\
&\P\Big(S(n+1)=1 \Big| S(n)=0\Big) =1.
\ea
Set
\bel{eq:L}
L(n):=\sum_{k=0}^{n-1} \1(X(k)=0)= \sum_{k=0}^{n-1} \1(S(k)=0)=\sum_{j\in U}(L^+_j(n) +L^-_j(n))
,\quad  n\geq 0,
\ee
$L(n)$ being the number of visits of $S$ to the origin
up to time $n$.
\begin{lem}\label{lem:refl}
We have the weak convergence
\ba
\Big(\frac{S([n \cdot ])}{\sqrt{n}},  \frac{L([n \cdot ])}{\sqrt{n}}\Big)
\Rightarrow \frac{1}{\sqrt m}
\Big( |B(\cdot)|,L^0(\cdot ) \Big), \ n\to\infty,
\ea 
where $B$ is a standard Brownian motion, $L^0(\cdot)$ is a local time of $B$ at 0 defined by \eqref{e:loctime}.
\end{lem}
\begin{proof}
Whereas the convergence of the reflected random walks to a reflected Brownian motion is straightforward, certain 
work should be done to ensure the convergence of the local times.
A similar result in a more general setting can be found in Section 2 of Pilipenko and Prykhodko \cite{pilipenko2014jump} so that here we just briefly 
outline the argument.

First we construct an auxiliary slowed-down one-dimensional symmetric random walk $Q$ on $\mbZ$ with steps $\pm 1$ and $0$, namely we set
\ba
\P\Big(Q(n+1)&=j- 1 \Big| Q(n)=j\Big) =\P\Big(Q(n+1)=j+1 \Big| Q(n)=j\Big) =\frac{1}{2m}, \\
\P\Big(Q(n+1)&=j  \Big| Q(n)=j\Big)  =1- \frac{1}{m}, \quad  j\in\mbZ.
\ea
Clearly, by the functional central limit theorem, 
\ba
\label{e:Q}
Q([n\cdot ])/\sqrt n \Rightarrow B(\cdot)/\sqrt m,\quad n\to\infty. 
\ea
Then we construct a copy of the reflected random walk $S$ on the basis of the random walk $Q$. To this end, 
we write the solution of the Skorokhod reflection problem for $Q$,
\ba
\tilde Q(n)=Q(n)+ L^Q(n),
\ea
where
\ba
L^Q(n)=-\min_{0\leq k\leq n} (Q(k)\wedge 0).
\ea
Then $\tilde Q(n)\geq 0$. Since $\tilde Q$ may spend several consecutive steps at the origin, we perform a random time 
transformation with the help of the local time $L^Q$ to obtain a copy of a Markov chain $S$. Essentially we push the
the process $\tilde Q$ from zero with probability one. Since the number of visits of zero on the time interval $[0,n]$ is of the order 
$\sqrt n$ this random time transformation (properly rescaled) converges to identity. Since the reflection mapping is continuous, 
from the convergence \eqref{e:Q} we obtain convergence of the rescaled constructed copy of $S$ and the local time $L^Q$ to the solution of the
Skorokhod reflection problem for $B/\sqrt m$, namely to $(\max_{s\leq t}B_s-B_t, \max_{s\leq t}B_s)/\sqrt m$, which 
has the same law as $(|B(\cdot)|, L^0(\cdot))/\sqrt m$ by L\'evy's theorem, see, e.g.\ \cite[Corollary 19.3]{Kallenberg-02}. 
\end{proof}

\begin{lem}
\label{lem:compactness}
The sequence $\{(X^{j,\pm}_n(\cdot), M^{j,\pm}_n(\cdot), \nu^{j,\pm}_n(\cdot))_{j\in U}\}_{n\geq 1} $
is weakly relatively
 compact in $D([0,T], \mbR^{6|U|})$ and any limit point is a continuous processes.
\end{lem}
\begin{proof}
Notice that the modulus of continuity (in the uniform topology)
of any $X^{j,\pm}_n$ is dominated by doubled modulus of continuity of 
$\{S(n\cdot)/\sqrt{n}\}$, and 
the modulus of continuity 
of any $M^{j,\pm}_n$ is dominated by doubled modulus of continuity of 
$X^{j,\pm}_n$.
The third coordinate is the difference of the first two. Hence the proof of the Lemma follows from Lemma \ref{lem:refl}.
\end{proof}

\medskip
\noindent
\emph{Proof of Proposition \ref{p:WBM}}.
To show convergence $\{(X_n^{j,\pm})_{j\in U}\}_{n\geq 0}$ to the WBM it suffices 
to verify that for 
any subsequence $\{ (X_{n_k}^{j,\pm})_{j\in U}  \}_{k\geq 0}$  there is a subsubsequence that converges to the WBM.
Due to Lemma \ref{lem:compactness} without loss of generality we will assume that the sequence 
$\Big\{\Big(X^{j,\pm}_n(\cdot), M^{j,\pm}_n(\cdot), \nu^{j,\pm}_n(\cdot)\Big)_{j\in U} \Big\}_{n\geq 1} $
converges in distribution  to a continuous process 
$\Big(X^{j,\pm}(\cdot), M^{j,\pm}(\cdot), \nu^{j,\pm}(\cdot)\Big)_{j\in U}$.

Let us check the conditions a)--d) of Theorem~\ref{Walsh} for the process $\sqrt{m}\Big(X^{j,\pm}(\cdot), M^{j,\pm}(\cdot), \nu^{j,\pm}(\cdot)\Big)_{j\in U}$.

\medskip

\noindent 
a) It follows from the construction that $X^{j,\pm}(t)\geq 0$, $t\in[0,T]$. Moreover, 
 only one of these processes may be positive at any fixed time.
 
\medskip

\noindent  
b) and d) The processes $\nu^{j,\pm}(\cdot)$ are non-decreasing a.s.\ and $\nu^{j,\pm}(0)=0$.
Lemma \ref{lem:refl} yields that
\ba
\Big(|X(\cdot)|, \nu(\cdot)\Big) 
\stackrel{\di}{=}
\frac{1}{\sqrt{m}}
\Big( |B(\cdot)|,  L^0(\cdot)\Big),
\ea
where
\ba
\nu(t)= \sum_{j\in U}(\nu^{j,-}(t)+ \nu^{j,+}(t)).
\ea
Since $\int_0^T \1(B(t)=0)\, \di t =0$ and $\int_0^T \1(|B(t)|>0) \,\di L^0(t)=0$ a.s.\ for any $T>0$, we have 
\ba
&\int_0^T \1(|X(t)|=0)\,  \di t =0,\\
&\int_0^T \1(|X(t)|>0)\,  \di \nu(t)=0
\ea 
almost surely.

\medskip

\noindent  
c) It follows from the construction that
\ba
X^{j,\pm}(t)=M^{j,\pm}(t)+ \nu^{j,\pm}(t)\ \text{a.s.}
\ea
for all $j\in U, t\geq 0$.
To show that 
\ba
\label{eq:429}
\nu^{j,\pm}(t)=\pi^\pm_j \nu(t)\ \text{a.s.}
\ea
we recall the strong law of large numbers \eqref{eq:statPi} for Markov chains. 
For any $t>0$ the process $L([nt])$ defined in \eqref{eq:L} increases to $+\infty$ a.s.\ 
as $n\to\infty$. Hence 
\ba
\frac{\nu_{n}^{j,\pm}(t)}{\sum_{k\in U}\nu_{n}^{k,\pm}(t)}=
\frac{1}{L([nt])}\sum_{k=0}^{L([n t])}\1(\hat \cY(k)= \{\pm 0\} \times j)\to \pi^\pm_j  \ \text{a.s.}
\ea
where   $\hat \cY$ is defined in  \eqref{e:hatY}.
 
The processes $M_{n}^{j,\pm}(\cdot)$, $j\in U$, are local martingales  with respect to 
the filtration generated by $\{X_{n}^{j,\pm}(\cdot), M_{n}^{j,\pm}(\cdot)\}_{j\in U}$.
Since the jumps of each $M_{n}^{j,\pm}(\cdot)$ are uniformly bounded, the limits $M^{j,\pm}$ are local martingales with respect to 
filtration generated by $\{X^{j,\pm}, M^{j,\pm}\}_{k\in U}$ due
to Lemma 1.17 in Chapter IX of Jacod and Shiryaev \cite{JacodS-03}. Moreover, the limit processes are continuous due to Lemma \ref{lem:compactness}.

It is left to show that
\ba
\label{eq:433}
\langle M^{j,\pm}\rangle_t= \frac{1}{  m}\int_0^t \1(X^{j,\pm}(s)>0) \,\di s \ \text{a.s.\ for }  j\in U,\ t\geq 0.
\ea
By Skorokhod's representation theorem there is a probability space and 
the copies
\ba
\Big\{\Big(\tilde X^{j,\pm}_n(\cdot),\tilde  M^{j,\pm}_n(\cdot),\tilde  \nu^{j,\pm}_n(\cdot)\Big)_{j\in U}\Big\}_{n\geq 1}
\text{ and }
\Big\{\Big(\tilde X^{j,\pm}(\cdot),\tilde  M^{j,\pm}(\cdot), \tilde  \nu^{j,\pm}(\cdot)\Big)_{j\in U}\Big\}
\ea
of 
\ba
\Big\{\Big(X^{j,\pm}_n(\cdot), M^{j,\pm}_n(\cdot), \nu^{j,\pm}_n(\cdot)\Big)_{j\in U}\Big\}_{n\geq 1}\text{ and } 
\Big\{\Big(X^{j,\pm}(\cdot), M^{j,\pm}(\cdot), \nu^{j,\pm}(\cdot)\Big)_{j\in U}\Big\}
\ea
such that on any interval $[0,T]$
we have a.s.\ uniform convergence
\bel{eq:Skor}
  \Big(\tilde X^{j,\pm}_n(\cdot), \tilde  M^{j,\pm}_n(\cdot), \tilde  \nu^{j,\pm}_n(\cdot)\Big)_{j\in U}   
\to
 \Big(\tilde X^{j,\pm}(\cdot), \tilde M^{j,\pm}(\cdot), \tilde \nu^{j,\pm}(\cdot)\Big)_{j\in U},\quad  n\to\infty.
 \ee
To prove \eqref{eq:433} 
it suffices to verify that with probability 1 the sequence 
\ba
\Big(\tilde M^{j,\pm}_n(t)\Big)^2- \frac1m \int_0^{[nt]/n}  \1(\tilde X^{j,\pm}_n(s)>0)\,\di s
\ea
converges uniformly over $t\in [0,T]$ to 
\ba
(\tilde M^{j,\pm}(t))^2- \frac1m\int_0^{t}  \1(\tilde X^{j,\pm}(s)>0)\, \di s.
\ea
Here we again use Lemma 1.17 in Chapter IX of Jacod and Shiryaev \cite{JacodS-03} and a localization procedure.

It follows from \eqref{eq:Skor} that we have to prove the convergence of the integrals only. 

Let $\omega\in \Omega$ be such that \eqref{eq:Skor} holds. If $s\in[0,T]$ is such that $\tilde X^{k,\mathfrak{s}}(s)>0$ for some 
$1\leq k\leq |U|$ and 
$\mathfrak{s}\in\{-,+\}$
then    
$\tilde  X^{k,\mathfrak{s} }_n(s)>0$ for large $n$. Since only one of the processes
$\{\tilde X^{j,\mathfrak{s}}(s)\}_{j\in U,\mathfrak{s}\in\{-,+\}}$ and only one of 
$\{\tilde X^{j,\mathfrak{s}}_n(s)\}_{j\in U,\mathfrak{s}\in\{-,+\}}$  may be non-zero
we have
convergence of the indicators for all $j\in U$:
\bel{eq:7030}
\lim_{n\to\infty} \1(\tilde X^{j,\pm}_n(s)>0) = \1(\tilde X^{j,\pm}(s)>0).
\ee

The process $\tilde X= (\tilde X^{j,+},\tilde X^{j,-})_{j\in U}$ spends zero time in 0 with probability 1 because 
$\sqrt{m}\sum_{j=1}^{|U|}(\tilde X^{j,+}+\tilde X^{j,-} ) $ is a reflected Brownian motion, see
 Lemma \ref{lem:refl}. Therefore
for a.a.\ $\omega$ and a.a.\ $s\in[0,T]$ there is $k$ and $\mathfrak{s}$ such that $ \tilde X^{k,\mathfrak{s}}(s)>0$ 
and we have \eqref{eq:7030} for any index $(j,\pm)$, $j\in U$. So by the Fubini theorem and by the 
Lebesgue dominated convergence theorem for a.a.\ $\omega$ and all $j\in U$ we have convergence of the integrals
\ba
\lim_{n\to\infty} \int_0^t\1(\tilde X^{j,\pm}_n(s)>0)\,\di s = \int_0^t\1(\tilde X^{j,\pm}(s)>0)\,\di s.
\ea
 This completes the proof of Proposition \ref{p:WBM}.\hfill $\Box$
 
\medskip

To treat convergence of $Y$, similarly to the representation \eqref{eq:X_j_representation} for $\{X^{j,\pm}(n)\}$ we 
decompose the sequence $Y=\{Y(n)\}_{n\geq 0}$ into the sum
\ba
Y(n)&=\sum_{k=0}^{n-1} \Big(Y(k+1) -Y(k)\Big)\1(X(k)\neq 0)\\
&+\sum_{j \in U} \sum_{k=0}^{n-1} \Big(Y(k+1) -Y(k)\Big)\1(\cX(k)= \pm 0, Y(k)\equiv j)\\
&= M^Y(n) +D^Y(n),\quad n\geq 0,
\ea
and we define
\ba
M^Y_n(t):=\frac{M^Y([nt])}{\sqrt{n}},\quad  D^Y_n(t):=\frac{D^Y([nt])}{\sqrt{n}}.
\ea

The following Proposition is proven analogously to the previous reasoning.
\begin{prp}
\label{p:Y}
Let $X$ and $\nu$ be as in Proposition \ref{p:WBM}. 
Then
 \ba
\Big(X_n(\cdot), \nu_n(\cdot),  M^Y_n(\cdot), D^Y_n(\cdot)\Big)
\Rightarrow 
\Big(X(\cdot), \nu(\cdot),  \frac{1}{\sqrt m} W_Y(\cdot), c \nu(\cdot)\Big),
\ea
where $W_Y$ is a $(m-1)$-dimensional Brownian motion 
independent of $X$, and $c$ is defined in \eqref{e:c}.
\end{prp}

\medskip
\noindent
\emph{Proof of Theorem \ref{thm:periodic_membrane}.} 
We combine Propositions \ref{p:WBM} and \ref{p:Y} together with Corollary \ref{c:SBM}.
\hfill  $\Box$

\section{One-sided membrane with ergodic properties \label{s:rm}}

The same method of decomposition of the perturbed Markov chain into a sum of excursions
 combined with the strong law of large numbers 
\eqref{eq:statPi} can be applied for the analysis of a one-sided membrane 
that has ergodic properties. 

As in Section \ref{s:twosided}, let $m\geq 2$ and let $\{\ex_1,\dots,\ex_m\}$ be a 
standard basis in $\mbR^m$. 
Consider a Markov chain $Z=(X,Y)$ on $\mbZ^m$ that behaves as 
a simple random walk outside of the hyperplane $H:=\{0\}\times \mbZ^{m-1}$, i.e.,
\ \eqref{eq:transition} holds true.

Let $\{p_y\}_{y\in H}\subset [0,1]$.
Now we interpret $H$ as a semipermeable non-homogeneous 
 membrane that may let a particle into one half-space with 
  probabilities $\{p_y\}_{y\in H}$. More precisely, we assume that
for each $z=(0,y)  \in H$:
\ba
p_y =\P \Big(Z(n+1)=z+\ex_1 \Big|  Z(n)=z\Big)=1-\P \Big(Z(n+1)=z-\ex_1\, \Big|
 \, Z(n)=z\Big).
\ea
Note that the particle leaves the membrane in the direction orthogonal to $H$, i.e.,\ $Y(n+1)=Y(n)$ for $Z(n)\in H$ with probability 1, and hence there is 
no slide along the membrane.

We assume that  the membrane has the following  ergodic property:

\noindent
\textbf{A}$_\text{SLLN}(\beta)$: there is $\beta>0$ such that 
\ba
\lim_{A\to\infty}  \frac{1}{|V(A,o)|}\sum_{y\in V(A,o)} p_y  = :\bar p\in[0,1],
\ea 
where the limit is taken over all cubes $V(A,o)\subseteq H$ of volume $|V(A,o)|$
with side size larger than $A$ and whose centre $o$ is within distance $A^\beta$ from the origin.

We give two clarifying examples for the assumption \textbf{A}$_\text{SLLN}(\beta)$. 

\begin{exa} 
Assume that the family $\{p_y\}_{y\in H}$ has a periodic structure: there are $k_2, \dots, k_m\geq 1$ such that 
for all  $l_2,\dots,l_m\in\mbZ$ and for all $y\in H$
\ba
p_y=p_{y+ k_2 l_2 \ex_2+\cdots +k_m l_m \ex_m}.
\ea
Then $\{p_y\}$ clearly satisfy assumption  \textbf{A}$_\text{SLLN}(\beta)$ for any $\beta>0$  with
\ba
\bar p= \frac{1}{k_2 \cdots k_m} \sum_{i_2,\dots,i_m=1}^{k_2,\dots,k_m} p_{(i_2,\dots,i_m)}. 
\ea
\end{exa}
\begin{exa} 
Let $\{p_y\}_{y\in H}$ be i.i.d.\ random variables with values in $[0,1]$ defined on a probability space $(\Omega', \rF', \P')$. Then for each fixed $\omega'\in\Omega'$
the family $\{p_y(\omega')\}$  defines a random ``environment''. Then
the assumption  \textbf{A}$_\text{SLLN}(\beta)$ with  $\beta>0$ is satified with
\ba
\bar p=\E' p_y.
\ea
To see this, let $\beta>0$, and let $V(A,o)$ denote a cube with the size $A\in\mbN$ and the centre at $o\in \mbZ^{m-1}$. Then
\ba
\label{e:prob}
\P\Big(  \frac{1}{|V(A,o)|}\sum_{y\in V(A,o)} p_y \not\to \bar p   \Big)
=\P\Big( \bigcup_{m=1}^\infty \bigcap_{k=1}^\infty\bigcup_{A=k}^\infty \bigcup_{|o|\leq L^\beta}
\Big\{
\frac{1}{|V(A,o)|}\Big|\sum_{y\in V(A,o)} (p_y-\bar p)\Big| >\frac{1}{m} 
\Big\}
\Big).
\ea
For each $m\geq 1$ and $A\geq 1$ we apply Hoeffding's inequality, see Chapter III, \S 5.8 in Petrov \cite{petrov75sums}:
\ba
\P\Big( \bigcup_{|o|\leq A^\beta}
\Big\{
\frac{1}{|V(A,o)|}&\Big|\sum_{y\in V(A,o)} (p_y-\bar p)\Big| >\frac{1}{m} 
\Big\}
\Big)
\leq 
(2A^\beta)^{m-1}
\P\Big(
\frac{1}{|V(A,o)|}\Big|\sum_{y\in V(A,o)} (p_y-\bar p)\Big| >\frac{1}{m} 
\Big)\\
&\leq 
2 (2A^\beta)^{m-1}   \ex^{-2 |V(A,o)|/m^2}= 2 (2A^\beta)^{m-1}   \ex^{-2 A^{m-1}/m^2}.
\ea
Hence, the probability in \eqref{e:prob} equals to 0.
\end{exa}

\begin{thm}
\label{thm1}
Let assumption \emph{\textbf{A}$_\text{SLLN}(\beta)$} holds true for $\beta>1$. Then for any initial value $Z(0)\in\mbZ^m$ the
weak convergence holds true:
\ba
(X_n,Y_n)\Rightarrow \frac{1}{\sqrt m} ( X^{\gamma}, Y^0 ),\quad n\to\infty,
\ea
where $X^\gamma$ and $Y^0$ are defined in \eqref{eq:Skew} and \eqref{eq:Y},
with $\gamma = 2\bar p-1$, and $c=0$.
\end{thm}
\begin{proof}
1. Without loss of generality assume that $Z(0)=0$. As in Section \ref{s:pm} we 
decompose the Markov chain $Z=(X,Y)$ into the ``left'' and the ``right'' excursions. Since the membrane is one-sided, and there is no need of 
introducing the set $U$, the notation of Section \ref{s:pm} simplifies significantly. Similarly to \eqref{eq:X_j_notation} we define
\ba
X^{+}(n)&:=X(n)\cdot \1(X(n)>0),\\
X^{-}(n)&:=|X(n)|\cdot \1(X(n)<0),
\ea
so that
\ba
X(n)=X^+(n)-X^-(n),\quad n\geq 0.
\ea
Then we decompose these processes similarly to \eqref{eq:X_j_representation} as
\ba
X^{\pm}(n)&=\sum_{k=1}^{n} \Big(X^{\pm}(k)-X^{\pm}(k-1)\Big) \cdot \1(X^{\pm}(k-1)>0) + 
\sum_{k=1}^{n}\1(X^\pm(k-1)=0, X^\pm(k)=1)\\
&=: M^{ \pm}(n) + L^{ \pm}(n),\qquad
M^{ \pm}(0) = L^{ \pm}(0)=0.
\ea
The processes $M^\pm$ are martingales, and $L^\pm$ are non-decreasing processes.
Recall the processes $X_n(\cdot)$ and $Y_n(\cdot)$ defined in \eqref{e:XY}, and define additionally the scaled processes 
$X^{\pm}_n$, $M^{\pm}_n$, $\nu^{\pm}$ similarly to \eqref{e:scaled} (omitting the index $j$), so that
\ba
X_n(t)=X^+(t)-X^-(t),\quad M_n(t)=M^+(t)-M^-(t),\quad \nu_n(t):=\nu^{+}_n(t)+ \nu^{-}_n(t).
\ea
Analogously to reasoning of the previous section (Propositions \ref{p:WBM} and \ref{p:Y}), we
have that
the sequence 
\ba
\{(X^{+}_n, X^{-}_n, X_n, M^{+}_n, M^{-}_n, M_n,  \nu^{+}_n,\nu^-_n, \nu_n,Y_n)\}_{n\geq 0} 
\ea
is weakly relatively compact
in $D(\bR_+,\bR^{10})$
each its limit point $(X^{+}, X^{-}, X, M^{+}, M^{-}, M,  \nu^{+},\nu^-, \nu,Y) $ is
continuous, and
\ba
X(t)=X^+(t)- X^-(t), \quad M(t)=M^+(t)- M^-(t),\quad \nu(t)=\nu^+(t)+ \nu^-(t),
\ea
where $M^+$  and $M^-$ are local martingales with the brackets 
\ba
\langle M^{\pm}\rangle_t=\frac{1}{m}\int_0^t \1(X^\pm(s)>0)\,\di s.
\ea
Moreover, the process $\sqrt{m}|X|=\sqrt{m}(X^+ +X^-)$ is a standard reflected Brownian motion, 
$\sqrt{m} \nu$ is its local time at 0, and the process $\sqrt{m} Y$ is a standard $(m-1)$-dimensional Brownian motion independent of $X$.

Hence to prove Theorem \ref{thm1} it remains   to show that 
\ba
\label{e:conv}
\frac{\nu^+(t)}{\nu(t)}
=\bar p\quad \text{a.s.\ for all $t> 0$}.
\ea
To verify this, it suffices to prove that 
\ba
\label{e:conv1}
\lim_{n\to\infty} \frac{\nu^+_n(t)}{\nu_n(t)}=
 \lim_{n\to\infty} \frac{L^+([nt])}{L([nt])}
=\bar p\quad \text{a.s.\ for all $t> 0$}.
\ea
\medskip
\noindent 
2. Let $\tau_k$ be the moment of $k$th visit of $X$ to 0, $k\geq 0$, $\tau_0=0$, so that $Z(\tau_k)=(0,Y(\tau_k))$. 
By the strong Markov property we have
\ba
\Law(\tau_k-(\tau_{k-1}+1))=\Law(\tau|X(0)=1),\quad k\geq 1,
\ea
and 
\ba
\Law(Y(\tau_k)-Y(\tau_{k-1}))=\Law(Y(\tau)|X(0)=1,Y(0)=0),\quad k\geq 1,
\ea
where $\tau$ is the first return time to zero of $X$,
\ba
\tau=\inf\{k\geq 1\colon X(k)=0\}.
\ea
It is well known that  $\tau_k<\infty$ a.s.\ and $\tau_k\to+\infty$ a.s.\ as $k\to\infty$. 
Hence 
\ba
 Y(\tau_n)= \sum_{k=1}^{n}\Big(Y(\tau_k)-Y(\tau_{k-1})\Big),\ \  {n\geq 0},
\ea
 is a random walk on $\mbZ^{m-1}$
whose jumps have the probability distribution $\Law(Y(\tau)|X(0)=1,Y(0)=0)$.

We claim that
\ba
\label{e:stable}
\frac{Y(\tau_n)}{n} \stackrel{\di}{\to} S,\quad n\to\infty,
\ea
where $S$ is a $1$-stable random variable on $\bR^{m-1}$ with the characteristic function $\E \ex^{\i \langle u ,S\rangle}=\ex^{-|u|/\sqrt m}$, $u\in\bR^{m-1}$.
Indeed, for each $n\geq 1$ consider the symmetric random walk $\tilde Z=(\tilde X,\tilde Y)$ on $\mbZ^m$ starting at $\tilde Z(0)=(n,0,\dots,0)$. Denote 
\ba
\tilde \tau_0&=0,\\
\tilde \tau_1&=\inf\{k\geq 1\colon \tilde X(k)=n-1\},\\
&\cdots\\
\tilde \tau_n&=\inf\{k\geq 1\colon \tilde X(n)=0\}.
\ea
Then clearly
\ba
(\tau_1,\dots,\tau_n, Y(\tau_1),\dots,Y(\tau_n))\stackrel{\di}{=}(\tilde \tau_1,\dots,\tilde \tau_n, \tilde Y(\tilde \tau_1),\dots,\tilde Y(\tilde \tau_n))
\ea
By the functional central limit theorem,
\ba
\Law\Big( \sqrt m\frac{\tilde Z([n^2\, \cdot ])}{n}\Big| \tilde Z(0)=(n,0,\dots,0) \Big)\Rightarrow \Law \Big(W(\cdot)\Big|W(0)=(1,0\dots,0)\Big),
\ea
where $W$ is a standard $m$-dimensional Brownian motion, and thus
\ba
\Big(n^2 \tilde \tau_n, n^{-1}\sqrt m \tilde Y( \tilde \tau_n)\Big)\stackrel{\di}{\to} \Big(
\tau^W, \big(W_2(\tau^W),\dots,W_m(\tau^W)\big)\Big),
\ea
where $\tau^W=\inf\{t\geq 0\colon W_1(t)=0\}$. It is well known that $(W_2(\tau^W),\dots,W_m(\tau^W))$ is a 1-stable random vector,
see, e.g.,\ Theorem II.1.16 in Bass \cite{bass1995probabilistic}.

\medskip
\noindent 
3. With \eqref{e:stable} in hand, we apply Theorem 1.3 from Dolgopyat et al.\ \cite{dolgopyat2021global}, that states that 
under the assumption \textbf{A}$_\text{SLLN}(\beta)$ with $\beta>1$
\ba
\frac{1}{n}\sum_{k=0}^{n-1} p_{Y(\tau_k)}\to \bar p\ \  \text{a.s.},\quad n\to\infty.  
\ea

\medskip
\noindent 
4. Now we are able to finish the proof of \eqref{e:conv}
Then
\ba
L^{+}(n)&=\sum_{k=1}^{n}\1(X(k-1)=0, X(k)=1)=\sum_{k\colon \tau_k\leq n-1} \1(X(\tau_k+1)=1),\\
L(n)&=\sum_{k=1}^{n}\1(X(k-1)=0)=\max\{k\geq 1\colon \tau_k\leq n-1\}.
\ea
Since $\tau_k\to+\infty$ a.s., $L(n)\to\infty$ a.s.\ as well as at least one of the processes $L^+(\cdot)$ and  $L^-(\cdot)$.

Consider the process
\ba
Q^+(n):=\sum_{k=0}^{n-1} \1(X(\tau_k+1)=1),\\
\ea
and note that
\ba
\lim_{n\to\infty}\frac{L^{+}(n)}{L(n)}=\lim_{n\to\infty}\frac{Q^{+}(n)}{n}
&=\lim_{n\to\infty}\frac{1}{n}\sum_{k=0}^{n-1} \Big(\1(X(\tau_k+1)=1) - p_{Y(\tau_k)}\Big) +  
\lim_{n\to\infty}\frac{1}{n}\sum_{k=0}^{n-1} p_{Y(\tau_k)}\\
&=\lim_{n\to\infty}\frac{1}{n}\sum_{k=0}^{n-1} \Big(\1(X(\tau_k+1)=1) - p_{Y(\tau_k)}\Big)+ \bar p.
\ea
The process
\ba
V(n):=\sum_{k=0}^{n-1} \Big(\1(X(\tau_k+1)=1) - p_{Y(\tau_k)}\Big) 
\ea
is a martingale difference with
\ba
\E\Big[\1(X(\tau_k+1)=1)-p_{Y(\tau_k)}\Big]=0 \quad \text{and}\quad  \E\Big[\1(X(\tau_k+1)=1)-p_{Y(\tau_k)}\Big]^2\leq 2.
\ea
Hence by the strong law of large numbers for martingales (Theorem 8b in Chapter II, \S 3 of  Gikhman and Skorokhod \cite{GikSko-74}) we have convergence
\ba
\frac{V(n)}{n}\to 0 \ \ \text{a.s.}, \quad  n\to\infty,
\ea
what finishes the proof of Theorem \ref{thm1}.
\end{proof}

%
%
%

\end{document}